\documentclass{article}

\usepackage{arxiv}

\usepackage[utf8]{inputenc} % allow utf-8 input
\usepackage[T1]{fontenc}    % use 8-bit T1 fonts
\usepackage{hyperref}       % hyperlinks
\usepackage{url}            % simple URL typesetting
\usepackage{booktabs}       % professional-quality tables
\usepackage{amsfonts}       % blackboard math symbols
\usepackage{nicefrac}       % compact symbols for 1/2, etc.
\usepackage{microtype}      % microtypography
\usepackage{lipsum}		% Can be removed after putting your text content

\usepackage{tikz}                        % For commutative diagrams
\usetikzlibrary{matrix,arrows}
\usepackage{amsmath}
\usepackage{amsthm}
\usepackage{amssymb}
\usepackage{mathtools}                  % For the big | 
\usepackage{dsfont}                     % For the identity matrix
\usepackage{amsxtra}                     % Use various AMS packages
\usepackage{epsfig}
\usepackage{verbatim}
\usepackage{epigraph}                    % For epigraphs...
\DeclarePairedDelimiter\norm{\lVert}{\rVert}%
\DeclareMathOperator{\id}{id}

\theoremstyle{plain}
\newtheorem{thm}{Theorem} %Hier und im Folgendem theorem -->thm
\newtheorem{prop}[thm]{Proposition}
\newtheorem{corollary}[thm]{Corollary}
\newtheorem{lemma}[thm]{Lemma}

\theoremstyle{definition}
\newtheorem{dfn}[thm]{Definition}

\newtheorem{example}[thm]{Example}

\theoremstyle{remark}
\newtheorem{remark}[thm]{Remark}
\newcommand{\bigslant}[2]{{\raisebox{.2em}{$#1$}\left/\raisebox{-.2em}{$#2$}\right.}}
\newcommand{\be}{\begin{equation}}
\newcommand{\ee}{\end{equation}}
\newcommand{\bdm}{\begin{displaymath}}
\newcommand{\edm}{\end{displaymath}}
\newcommand{\ba}[1]{\begin{array}{#1}}
\newcommand{\ea}{\end{array}}

\newcommand{\bcen}{\begin{center}}
\newcommand{\ecen}{\end{center}}
\newcommand{\btab}{\begin{tabular}}
\newcommand{\etab}{\end{tabular}}

\newcommand{\ra}{\rightarrow}

\newcommand{\vertiii}[1]{{\left\vert\kern-0.25ex\left\vert\kern-0.25ex\left\vert #1 
		\right\vert\kern-0.25ex\right\vert\kern-0.25ex\right\vert}}
	
\newcommand{\PM}{\ensuremath{\mathcal{P}}} % Menge der W'Maße
\newcommand{\B}{\ensuremath{\mathcal{B}}} % Borel-Sigma-Algebra
\newcommand{\R}{\ensuremath{\mathbb{R}}}
\newcommand{\N}{\ensuremath{\mathbb{N}}}

\title{Differentiable maps between Wasserstein spaces}

\date{October 5, 2020}	% Here you can change the date presented in the paper title

\author{
	Bernadette Lessel \\
	Max Planck Institute for the History of Science\\
	Berlin, Germany \\
	\texttt{blessel@mpiwg-berlin.mpg.de} \\
	\And
	Thomas Schick \\
	Mathematical Institute\\
	University of Göttingen\\
	\texttt{thomas.schick@math.uni-goettingen.de} \\
}

\begin{document}
\maketitle

\begin{abstract}
	A notion of differentiability for maps $F:W_2(M)\longrightarrow W_2(N)$ between Wasserstein spaces of order 2 is being proposed, where $M$ and $N$ are smooth, connected and complete Riemannian manifolds.
	Due to the nature of the tangent space construction on Wasserstein spaces, we only give a global definition of differentiability, i.e.\ without a prior notion of pointwise differentiability. With our definition, however, we recover the expected properties of a differential. Special focus is being put on differentiability properties of maps of the form $F=f_\#$, $f:M\longrightarrow N$ and on convex mixing of differentiable maps, with an explicit construction of the differential.
\end{abstract}

% keywords can be removed
%\keywords{Wasserstein spaces \and Optimal transport}

\tableofcontents

\section{Introduction}
%---------------------
Fundamental work has been done on the weak Riemannian manifold structure and second order analysis on Wasserstein spaces $W_2(M)$, most notably by Felix Otto \cite{Otto2001}, John Lott \cite{Lott_2007} and Nicola Gigli \cite{Gigli}. However, to our knowledge, no notion of differentiability for maps between Wasserstein spaces has been proposed in the literature yet. 

We begin with a reminder of Wasserstein spaces and its weak differentiable structure, to motivate the definitions we make later on. 
Our notion of differentiability for maps $F:W_2(M)\rightarrow W_2(N)$ between Wasserstein spaces is a global one, in the sense that it does not use a pointwise notion of differentiability. It seems to be the case that the latter is not possible in an immediate way due to the way tangent spaces are constructed in Wasserstein geometry: The basis for talking about tangent vectors along curves in $W_2(M)$ is constituted by the weak continuity equation $\partial_t\mu_t+\nabla\cdot(v_t\mu_t)=0$, which can be seen as a differential characterization of absolutely continuous curves in $W_2(M)$ (see Theorem \ref{them.ac}). The curve of minimal vector fields $v_t$ that solves the continuity equation for an absolutely continuous curve $\mu_t$ is then seen as being tangential along $\mu_t$. However, $v_t$ is only defined for almost every $t$, so that a pointwise evaluation is not meaningful and therefore undermines the definition of a pointwise notion of differentiability in our approach. The differential of a map is, however, defined in a pointwise manner. 

Our account on differentiable maps between Wasserstein spaces begins with the definition of \emph{absolutely continuous maps} which map absolutely continuous curves to absolutely continuous curves. This definition is made in analogy to the theorem in differential geometry that a map $f:M\rightarrow N$ is differentiable if and only if it maps differentiable curves to differentiable curves. Absolutely continuous maps serve as a pre-notion to differentiability. 
An absolutely continuous map $F:W_2(M)\rightarrow W_2(N)$ is then said to be \emph{differentiable} if every $\mu\in W_2(M)$ there exists a bounded linear map $dF_{\mu}$ between the tangent space at $\mu$ and the tangent space at $F(\mu)$ such that for every absolutely continuous curve $\mu_t$ the image curve $dF_{\mu_t}(v_t)$ of the curve of tangent vector fields $v_t$ along $\mu_t$ is a curve of tangent vector fields along $F(\mu_t)$ (Definition \ref{dfn.diff}). The collection of all these $dF_{\mu}$, in the sense of a bundle map between tangent bundles, is then called the \emph{differential} $dF$ of $F$.

We show that $dF$ unique up to a redefinition on a \emph{negligible} set. Also, the usual properties of the differential are derived, such as the expected differential of the constant and of the indentity mapping, also of the composition of two differentiable maps and of the inverse of a differentiable map. 

Special attention is payed to maps of the form $F=f_\#$, where measures are mapped to their image-measure with respect to $f:M\rightarrow N$, $f$ being smooth and proper and where $\sup_{x\in M}\norm{df_x}<\infty$. Maps of this kind are absolutely continuous, and an explicit formula is derived for a curve of vector fields satisfying the continuity equation together with $F(\mu_t)$, where $\mu_t$ is absolutely continuous.  
Unfortunately, it is not true in general that this curve of vector fields is actually tangent to $\mu_t$, i.e. minimal. To enforce that, one can, however, apply a projector onto the respective tangent spaces, for almost every $t$, which in particular guarantees the existence of a differential for $F$.

Further focus is being put on the treatment of differentiability properties of convex mixings of maps between Wasserstein spaces, as they provide a class of non-trivial maps which are not given by a pushfoward of measures.

For background knowledge on Wasserstein geometry and optimal transport we refer to \cite{Ambrosio2013} and \cite{Vilbig}.

%-------------------------------------------
\section{Wasserstein geometry}\label{ch.was}
%-------------------------------------------
%
Wasserstein geometry is a dynamical structure on Wasserstein spaces, which basically are sets of probability measures together with the Wasserstein distance. 

Let thus $(X,d)$ be a Polish space, where $d$ metrizes the topology of $X$, and $\PM(X)$ the set of all probability measures on $X$ with respect to the Borel $\sigma$-algebra $\B(X)$. Instead of $(X,d)$ we will often just write $X$. A measurable map between two Polish spaces $T:X\rightarrow Y$ induces a map between the respective spaces of probability measures via the \emph{pushforward} $T_\#$ of measures: $T_\#:\PM(X)\rightarrow\PM(Y)$, $\mu\mapsto T_\#\mu$, where $T_\#\mu(A):=\mu(T^{-1}(A))$, for $A\in\B(Y)$. The \emph{support} of a measure $\mu$ is defined by $supp(\mu):=\{x\in X\mid \text{ every open neighbourhood of $x$ has positive $\mu$-measure}\}$. The Lebesgue measure on $\R^n$ is denoted by $\lambda$. 
%
%------------------------------------
\subsection{Wasserstein spaces $W_p(X)$}
%------------------------------------
%

We denote the set of probability measures which have finite $p$-th moment by $\PM_p(X)$, where $p\in[1,\infty)$:
$$\PM_p(X):=\{\mu\in\PM(X)\ \mid \int_X d^p(x_0,x)\ d\mu(x)<\infty\}.$$

\noindent Note that $\PM_p(X)$ is independent of the choice of $x_0\in X$.
Furthermore, we define
\begin{equation}\nonumber
Adm(\mu,\nu):=\lbrace \gamma\in\PM(X\times Y)\mid\pi^X_\#\gamma=\mu, \pi^Y_\#\gamma=\nu\rbrace,
\end{equation}
the so called \emph{admissible transport plans} between $\mu$ and $\nu$. Here, $\pi^X:X\times Y\rightarrow X$, $\pi^X(x,y)=x$, similarly $\pi^Y$.

\begin{dfn}[\emph{\textbf{Wasserstein distances and Wasserstein spaces}}]
	Let $(X,d)$ be a Polish space and $p\in[0,\infty)$, then
	\begin{eqnarray*}\nonumber
		W_p: \PM_p(X)\times\PM_p(X)&\rightarrow& X\\\nonumber
		(\mu,\nu) &\mapsto& \left(\inf_{\gamma\in Adm(\mu,\nu)}\int_{X\times X}d^p(x,y)\ d\gamma(x,y)\right)^{1/p}
	\end{eqnarray*}
	is called the \emph{{$p$-th Wasserstein distance}}, or Wasserstein distance of order $p$. The tuple $(\PM_p(X),W_p)$ is called \emph{Wasserstein space} and is denoted by the symbol $W_p(X)$. 
\end{dfn}
The fact that $W_p$ is indeed a metric distance is a problem treated in optimal transport, where it is established that a minimizer for $$\inf_{\gamma\in Adm(\mu,\nu)}\int_{X\times X}d^p(x,y)\ d\gamma(x,y)$$ actually exists. Such a minimizer is called \emph{optimal transport plan}. In case a plan $\gamma\in Adm(\mu,\nu)$ is induced by a measurable map $T:X\rightarrow Y$, i.e. in case $\gamma=(Id,T)_\#\mu$, $T$ is called \emph{transport map}. Then, $T_\#\mu=\nu$.

One can show that $W_p(X)$ is complete and separable. Furthermore, $W_p$ metrizes the weak convergence in $\PM_p(X)$.

\begin{dfn}[\emph{\textbf{Weak convergence in $\PM_p(X)$}}]
	A sequence $(\mu_k)_{k\in\N}\subset\PM(X)$ is said to \emph{converge weakly} to $\mu\in\PM_p(X)$ if and only if $\int\varphi d\mu_k\ra\int\varphi d\mu$ for any bounded continuous function $\varphi$ on $X$. This is denoted by $\mu_k\rightharpoondown\mu$.
	A sequence $(\mu_k)_{k\in\N}\subset\PM_p(X)$ is said to converge weakly to $\mu\in\PM_p(X)$ if and only if for $x_0\in X$ it is:
	\begin{itemize}
		\item[1)] $\mu_k\rightharpoondown\mu$ and
		\item[2)] $\int d^p(x_0,x) d\mu_k(x)\ra\int d^p(x_0,x) d\mu(x)$.
	\end{itemize}
	This is denoted by $\mu_k\rightharpoonup\mu$.
\end{dfn}
An important class of curves in Wasserstein space that we will need later on are constant speed geodesics.
\begin{dfn}[\emph{\textbf{Constant speed geodesic}}]\label{df.geod}
	A curve $(\gamma_t)_{t\in [0,1]}$, $\gamma_0\neq\gamma_1$, in a metric space $(X,d)$ is called a \emph{constant speed geodesic} or \emph{metric geodesic} in case that
	\begin{equation}\label{eq.geod}
	d(\gamma_t,\gamma_s)=|t-s|d(\gamma_0,\gamma_1)\ \ \forall t,s\in [0,1]. 
	\end{equation}
\end{dfn}

\noindent We will often abbreviate curves $(\gamma_t)_{t\in [0,1]}$ by writing $\gamma_t$ instead.
\begin{dfn}[\emph{\textbf{Geodesic space}}]
	A metric space $(X,d)$ is called \emph{geodesic} if for every $x,y\in X$ with $x\neq y$, there exists a constant speed geodesic $\gamma_t$ with $\gamma_0=x$ and $\gamma_1=y$.
\end{dfn}
If $(X,d)$ is geodesic, then $W_2(X)$ is geodesic as well (\cite{Ambrosio2013}).
%
%------------------------------------------------------------------
\subsection{The continuity equation on $W_2(M)$}\label{sec.dyn}
%------------------------------------------------------------------
%
In the upcoming section, we will only be concerned with $W_2(M)$, where $M$ is a smooth, connected and complete Riemannian manifold with Riemannian metric tensor $h$ and associated Riemannian measure $\mu$. We will often write $W(M)$ instead of $W_2(M)$.
Furthermore, we equip the set of measurable sections of $TM$, which we will denote by $\Gamma(TM)$, with an $L^2$-topology. That means, for $v\in \Gamma(TM)$ we define
\begin{equation*}
	\|v\|_{L^2(\mu)}:=\sqrt{\int_M h(v,v)\ d\mu}
\end{equation*}
and  \begin{equation*}
	L^2(TM,\mu):=\bigslant{\{v\in \Gamma(TM)\mid \|v\|_{L^2(\mu)} <\infty\}}{\sim}.
\end{equation*}
Here, two vector fields are considered to be equivalent in case they differ only on a set of $\mu$-measure zero. $L^2(TM,\mu)$ is a Hilbert space with the canonical scalar product. We will often write $L^2(\mu)$ if it is clear to which manifold $M$ it is referred to.

The (infinite dimensional) manifold structure that is commonly used on $W(M)$ is not a smooth structure in the sense of e.g. \cite{Kriegl_1997} where infinite dimensional manifolds are modeled on convenient vector spaces. The differentiable structure on $W(M)$, that will be introduced below, rather consists of ad hoc definitions accurately tailored to optimal transport and the Wasserstein metric structure which only mimic conventional differentiable and Riemannian behavior.

Instead of starting with a smooth manifold structure, on Wasserstein spaces one starts with the notion of a tangent space. 
Traditionally, the basic idea of a tangent vector at a given point is that it indicates the direction a (smooth) curve will be going infinitesimally from that point. Then, the set of all such vectors which can be found to be tangent to some curve at a given fixed point are collected in the tangent space at that point. On $W(M)$, however, there is no notion of smooth curves. But there is a notion of metric geodesics. In case the transport plan for the optimal transport between two measures is induced by a map $T$, the interpolating geodesic on Hilbert spaces can be written as $\mu_t=((1-t)Id+tT)_\#\mu_0$, thus being of the form $\mu_t=F_{t\#}\mu_0$. More generally, on Riemannian manifolds optimal transport between $\mu_0$ and $\mu_t$ can be achieved by $\mu_t=F_{t\#}\mu_0$, $F_t=\exp(t\nabla\varphi)$( see e.g \cite{Vilbig}, Chapter 12). In these cases, $F_t$ is injective and locally Lipschitz for $0<t<1$ (\cite{Vilsmall}, Subsubsection 5.4.1).
It is known from the theory of characteristics for partial differential equations that curves of this kind solve the weak continuity equation, together with the vector field to which integral lines $F_t$ corresponds. 
\begin{dfn}[\emph{\textbf{Continuity equation}}]\label{dfn.conti}
	Given a family of vector fields $(v_t)_{t\in [0,T]}$, a curve $\mu_t:[0,T]\rightarrow W_2(M)$ is said to solve the \emph{weak continuity equation}
	\begin{equation}\label{eq.cont}
	\partial_t\mu_t+\nabla\cdot(v_t\mu_t)=0,
	\end{equation}
	if
	\begin{equation}\label{eq.wcont}
	\int_0^T\int_M\left(\frac{\partial}{\partial t}\varphi(x,t)+h(\nabla\varphi(x,t),v_t(x))\right)\ d\mu_t(x)dt =0
	\end{equation}
	holds true for all $\varphi\in C_c^\infty\left((0,T)\times M\right)$.
\end{dfn}
\begin{thm}[\emph{\textbf{\cite{Vilsmall}, Theorem 5.34}}]\label{thm.conti} % p.167
	Let $(F_t)_{t\in[0,T)}$ be a family of maps on $M$ such that $F_{t}:M\rightarrow M$ is a bijection for every $t\in[0,T)$, $F_0=Id$ and both $(t,x)\mapsto F_t(x)$ and $(t,x)\mapsto F_t^{-1}(x)$ are locally Lipschitz on $[0,T)\times M$. Let further $v_t(x)$ be a family of velocity fields on $M$ such that its integral lines correspond to the trajectories $F_t$, and $\mu$ be a probability measure. Then $\mu_t=F_{t\#}\mu$ is the unique weak solution in $\mathcal{C}\left([0,T),\PM(M)\right)$ of $\frac{d}{dt}\mu_t+\nabla\cdot(v_t\mu_t)=0$ with initial condition $\mu_0=\mu$. Here, $\PM(M)$ is equipped with the weak topology.
\end{thm}
%
	% A uniformly Lipschitz vector field generates a flow which is (globally) Lipschitz.
	
It is possible to characterize the class of curves on $W(M)$ that admit a velocity in the manner of Definition \ref{dfn.conti} (\cite{Ambrosio2013}) in the following way. 
\begin{dfn}[\emph{\textbf{Absolutely continuous curve}}]
Let $(E,d)$ be an arbitrary metric space and $I$ an interval in $\R$. A function $\gamma:I\rightarrow E$ is called \emph{absolutely continuous (a.c.),} if there exists a function $f\in L^1(I)$ such that 
\begin{equation}\label{eq.ac}
d(\gamma(t),\gamma(s))\leq\int_t^s f(r)dr,\ \ \ \forall s,t\in I, t\leq s.
\end{equation}
\end{dfn}
\begin{dfn}[\emph{\textbf{Metric derivative}}] \label{df.md}
The \emph{metric derivative} $|\dot{\gamma}|(t)$ of a curve $\gamma:[0,1]\rightarrow E$ at $t\in (0,1) $ is given as the limit
\begin{equation}
|\dot{\gamma}|(t)=\lim_{h\rightarrow 0}\frac{d(\gamma(t+h),\gamma(t))}{|h|}.
\end{equation}
\end{dfn}
Every constant speed geodesic is absolutely continuous and $|\dot{\gamma}|(t)=d(\gamma(0),\gamma(1))$.

\noindent It is known that for absolutely continuous curves $\gamma$, the metric derivative exists for a.e. $t$. It is an element of $L^1(0,1)$ and, up to sets of zero Lebesgue-measure, the minimal function satisfying equation \eqref{eq.ac} for $\gamma$. In this sense absolutely continuous functions enable a generalization of the fundamental theorem of calculus to arbitrary metric spaces.

\begin{thm}[\emph{\textbf{Differential characterization of a.c. curves}}]\label{them.ac}
Let $\mu_t:[0,1]\rightarrow W_2(M)$ be an a.c.~curve. Then there exists a Borel family of vector fields $(v_t)_{t\in[0,1]}$ on $M$ such that the continuity equation \eqref{eq.wcont} holds and 
\begin{equation*}
\|v_t\|_{L^2(\mu_t)}\leq |\dot{\mu_t}| \text{  for a.e. } t\in(0,1).
\end{equation*}
Conversely, if a curve $\mu_t:[0,1]\rightarrow W_2(M)$ is such that there exists a Borel family of vector fields $(v_t)_{t\in[0,1]}$ with $\|v_t\|_{L^2(\mu_t)}\in L^1(0,1)$, together with which it satisfies \eqref{eq.wcont}, then there exists an a.c. curve $\tilde{\mu}_t$ being equal to $\mu_t$ for a.e. $t$ and satisfying
\begin{equation*}
|\dot{\tilde{\mu}}_t|\leq\|v_t\|_{L^2(\tilde{\mu}_t)} \text{  for a.e. } t\in(0,1).
\end{equation*} 
\qed
\end{thm}
%
%Starting from geodesics, one arrives this way at the larger class of absolutely continuous curves on $W(M)$ and we would like to think about the vector fields $v_t$ as being tangential to $\mu_t$. To relate this to the standard notion of tangent vectors, one can show that if  $\mu_t$ is such that there is a smooth curve $\gamma:I\rightarrow \R^n$ and $\mu_t=\delta_{\gamma(t)}$, then $\mu_t$ solves the continuity equation in the weak sense with $v_t$ being such that $v_t(x)=\dot{\gamma}(t)$ in case $x=\gamma(t)$.
%
%---------------------------------------------------
\subsection{The tangent space $T_\mu W(M)$}
%---------------------------------------------------
As seen in Theorem \ref{them.ac}, every absolutely continuous curve in $W(M)$ admits an $L^1(dt)$-family of $L^2(\mu_t)$-vector fields $v_t$, i.e. $\|v_t\|_{L^2(\mu_t)}\in L^1(0,1)$, together with which the continuity equation is satisfied. In the following, we will call every such pair $(\mu_t,v_t)$ an \emph{a.c. couple}. We further want to call $v_t$ an \emph{accompanying vector field} for $\mu_t$.

Vector fields $v_t$ satisfying the continuity equation with a given $\mu_t$ are, however, not unique: there are many vector fields which allow for the same motion of the density: Adding another family $w_t$ with the ($t$-independent) property $\nabla(w_t \mu_t)=0$ to $v_t$  does not alter the equation. Theorem \ref{them.ac} provides a natural criterion to choose a unique element among the $v_t's$. According to this theorem, there is at least one $L^1(dt)$- family $v_t$ such that $|\dot{\mu}_t|=\|v_t\|_{L^2(\mu_t)}$ for almost all $t$, i.e. that is of minimal norm for almost all $t$. Linearity of \eqref{eq.ac} with respect to $v_t$ and the strict convexity of the $L^2$-norms ensure the uniqueness of this choice, up to sets of zero measure with respect to $t$. We want to call such a couple $(\mu_t,v_t)$, where $v_t$ is the unique minimal accompanying vector field for an a.c. curve $\mu_t$, a \emph{tangent couple}. 

It then seems reasonable to define the tangent space at point $\mu$ as the set of $v\in L^2(TM,\mu)$ with $\|v\|_\mu\leq\|v+w\|_\mu$ for all $w\in L^2(TM,\mu)$ such that $\nabla(w\mu)=0$. This condition for $v\in L^2(TM,\mu)$, however, is equivalent to saying that $\int_M h(v,w)\ d\mu=0$ for all $w\in L^2(TM,\mu)$ with $\nabla(w\mu)=0$. This in turn is equivalent to the following, which we will take as the definition of the tangent space.

\begin{dfn}[\emph{\textbf{Tangent space $T_\mu W(M)$}}] \label{dfn.tspace}
	The \emph{tangent space} $T_\mu W(M)$ at point $\mu\in W(M)$ is defined as
\begin{equation}\label{eq.tangentspace}
T_\mu W(M):=\overline{\{\nabla\varphi\mid \varphi\in \mathcal{C}^\infty_c(M)\}}^{L^2(TM,\mu)}\subset L^2(TM,\mu).
\end{equation}	
\end{dfn}

We also give the definition of the normal space:
\begin{eqnarray}\nonumber
T^\perp_\mu W(M) & := & \{w\in L^2(TM,\mu)\mid \int h(w,v)\ d\mu=0,\ \forall v\in T_\mu W(M)\}	\\\nonumber
& = & \{w\in L^2(TM,\mu)\mid \nabla(w\mu)=0\}.
\end{eqnarray}

\begin{remark}\label{cor.tan}
	If $(\mu_t,v_t)$ is an a.c. couple, then $(\mu_t, v_t)$ is a tangent couple if and only if $v_t\in T_{\mu_t}W(M)$ for almost every $t\in(0,1)$ (\cite{Gigli}, Proposition 1.30).
\end{remark}
It is not difficult to see that $\text{dim }T_{\delta} W(M)= \text{dim }M$, for a Dirac measure $\delta$, whereas in most of the cases $\text{dim }T_\mu W(M)=\infty$.
In general, it can be shown that as long as $\mu$ is supported on an at most countable set, $T_{\mu} W(M)=L^2(TM,\mu)$ (see \cite{Gigli}, Remark 1.33). Morally, the more points are contained in the support of the measure, the bigger gets the dimension. On the other hand, every probability measure can be approximated by a sequence of measures with finite support (see \cite{Vilbig} Thm 6.18), so that in each neighborhood of every measure there is an element $\mu$ with $\text{dim }T_\mu W(M)<\infty$. 

\noindent We call the disjoint union of all tangent spaces, 
\begin{eqnarray*}
TW(M):=\bigsqcup_{\mu\in W(M)}T_\mu W(M)=\bigcup_{\mu\in W(M)}\{(\mu,v)\mid v\in T_\mu W(M)\},
\end{eqnarray*}
 the \emph{tangent bundle of $W(M)$}. Since we are not treating $W(M)$ as a traditional manifold with charts, $TW(M)$ cannot be equipped with a traditional tangent bundle topology. Also, due to the denseness of the probability measures with finite support, local triviality cannot be achieved. However, since there is a natural projection map $\pi:TW(M)\rightarrow W(M);\ (\mu,v)\mapsto \mu$, we can in principle still talk about \emph{sections} and \emph{bundle maps} on the pointwise level. Whereas the notion of a vector field - in this context it would effectively be a field of (equivalence classes of) vector fields - has not turned out to be useful so far, we will use the concept of a bundle map later. In this spirit, a \emph{bundle map} between tangent bundles of Wasserstein spaces $W(M)$ and $W(N)$ is a fiber preserving map $B:TW(M)\rightarrow TW(N)$ in the sense that together with a continuous map $F:W(M)\rightarrow W(N)$ the commutativity of the following diagram is satisfied:
 \begin{center}
\begin{tikzpicture}
 \matrix(m)[matrix of math nodes,row sep=3em, 
 column sep=2.5em, text height=1.5ex, text depth=0.25ex]
 {TW(M) & & TW(N)\\
   W(M) & &  W(N)\\};
 \path[->, font=\scriptsize]
 (m-1-1) edge node[auto] {$B$} (m-1-3)
         edge node[left] {$\pi^M$} (m-2-1)
 (m-1-3) edge node[auto] {$\pi^N$} (m-2-3)
 (m-2-1) edge node[auto] {$F$} (m-2-3);
\end{tikzpicture}
\end{center}
One could ask about the meaningfulness of the condition that $F$ should be
continuous since for $B$ the concept of contiuity does not makes sense. It is just that we require the preservation of as much structure as possible. In any case, we are mainly going to use this idea of a bundle map to make clear how we want to see our notion of a differential of a differentiable function $F:W(M)\rightarrow W(N)$.

On $W(M)$ one can furthermore define a (formal) Riemannian structure. Intuition comes from the following formula which is due to J.-D. Benamou and Y. Brenier (\cite{Benamou1999}). It shows that the Wasserstein distance $W_2$, having been defined through the, static, optimal transport problem, can be recovered by a dynamic formula, being reminiscent of the length functional on Riemannian manifolds, defining the Riemannian metric distance.
\begin{thm}[\emph{\textbf{Benamou-Brenier formula}}]\label{thm.bb}
Let $\mu,\ \nu \in\PM_2(M)$, then
\begin{equation}\label{eq.bb}
W(\mu,\nu)=\inf_{(\mu_t,v_t)}\int_0^1\|v_t\|_{L^2(\mu_t)}\ dt,
\end{equation}
where the infimum is taken among all a.c. couples $(\mu_t,v_t)$ such that $\mu_0=\mu$ and $\mu_1=\nu$.
\end{thm}
%
%Because of formula (\ref{eq.bb}), we want to interpret the expression $\int_0^1\|v_t\|_{L^2(\mu_t)}\ dt$ as the length of a curve $\mu_t$. It can be shown (as stated in \cite{Vilbig}, Chapter 13) that the infimum is achieved if and only if $\mu$ and $\nu$ allow for an optimal plan. The minimizing curve will then be a tangent couple where $\mu_t$ is a geodesic.
%
This resemblance of formulas thus inspires the following definition.
\begin{dfn}[\emph{\textbf{Formal Riemannian tensor on $W_2(M)$}}]\label{df.Riem}
	The \emph{formal Riemannian metric tensor} $H_\mu$ on $W(M)$ at point $\mu\in W(M)$ is defined as
	\begin{eqnarray} \nonumber
	H_\mu:T_\mu W(M)\times T_\mu W(M) &\rightarrow & \R \\ \nonumber
	(v,w)                    &\mapsto & \int_M h_x(v,w)\ d\mu(x).
	\end{eqnarray}
\end{dfn}
Indeed, since $\|v_t\|_{L^2(\mu_t)}=\sqrt{\int_M h(v,v)\ d\mu}=\sqrt{H_\mu(v,v)}$, we now have $W(\mu,\nu)=\inf_{(\mu_t,v_t)}\int_0^1 \sqrt{H_\mu(v_t,v_t)}\ dt.$
The tuple $(T_\mu W(M)),H_\mu)$ constitutes a Hilbert Space. 

Gigli \cite{Gigithesis} emphasizes that Definition \ref{dfn.tspace} does not allow for a traditional Riemannian structure on $W_2(M)$ since the natural exponential map $v\mapsto \exp_\mu(v):=(Id+v)_\#\mu$ has injectivity radius $0$ for every $\mu$.

%-------------------------------------------------------
\section{Differentiable maps between Wasserstein spaces}\label{sec.diff}
%-------------------------------------------------------
%
Since $W_2(M)$ and $W_2(N)$ are not manifolds in a traditional sense, to be able to talk about differentiability of maps $F:W_2(M)\rightarrow W_2(N)$ we cannot compose $F$ with charts and apply Euclidean calculus. Recall, therefore, that a map $f:M\rightarrow N$ is differentiable if and only if it maps differentiable curves to differentiable curves. 
%
%-----------------------------------
\subsection{Absolutely continuous maps}
%-----------------------------------
%
Having only a notion of absolutely continuous curves, which are metrically differentiable almost everywhere and which are at the foundation of the construction of tangent spaces at Wasserstein spaces, we start with the following definition.
\begin{dfn}[\emph{\textbf{Absolutely continuous map}}]\label{def.ac}
A map $F:W(M)\rightarrow W(N)$ is called \emph{absolutely continuous}, or, \emph{a.c.}, if the curve $F(\mu_t)\subset W(N)$ is absolutely continuous up to redefining $t\mapsto\mu_t$ on a zero set, whenever $\mu_t\subset W(M)$ is absolutely continuous.
\end{dfn}
 We want to build our notion of differentiable maps between Wasserstein spaces on this idea of absolutely continuous maps. Before we continue to do so, we first find some conditions under which maps are absolutely continuous. For this, we want to recall the notion of \emph{proper} maps.
\begin{dfn}[\emph{\textbf{Proper map}}]\label{df.propermap}
	A continuous map $f:X\rightarrow Y$ between a Hausdorff space $X$ and a locally compact Hausdorff space $Y$ is called \emph{proper}, if for all compact subsets $K\subset Y$, the preimage $f^{-1}(K)\subset X$ is compact in $X$. 
\end{dfn}
In the following we denote the operator norm of a linear map by $\norm{\cdot}$.
\begin{thm}\label{thm.ac}
Let $F:W(M)\rightarrow W(N)$ be given as $F(\mu)=f_\#\mu$, $f:M\rightarrow N$ being smooth and proper and such that $\sup_{x\in M}\norm{df_x}<\infty$.
Then $F$ is absolutely continuous and for every tangent couple $(\mu_t,v_t)$, the tuple $(F(\mu_t),dF_{\mu_t}(v_t))$ is an a.c. couple, where

\begin{equation}\label{eq.vf}
dF_{\mu_t}(v_t)_y:= \int_{f^{-1}({y})}df_x(v_{t,x})\ d\mu_t^y(x) 
\end{equation}

for almost every $t$ and for $y\in f(M)$. Here, $df_x:T_xM\rightarrow T_{f(x)}N$ denotes the differential of $f$ at the point $x$, $v_{t,x}$ means the vector field $v_t$ at the point $x\in M$ and the probability measures $\mu_t^y(x)$ are defined through the disintegration theorem, $d\mu_t(x)=d\mu_t^y(x)df_\#\mu_t(y)$ (see Appendix \ref{A.dis}).\footnote{Note that what in \ref{A.dis} appears as lower index $y$, now appears as upper index $y$ since here were are additionally dealing with the $t$-dependence of $\mu_t$.} For all $y\notin f(M)$, we set $dF_{\mu_t}(v_t)_y=0$.

\end{thm}
Although $df_x:T_xM\rightarrow T_{f(x)}M$ is well defined for every $x$ as a mapping between tangent spaces, it is not well defined as a mapping between vector fields as long as $f$ is not injective. We thus take the mean value over all the vectors $df_x(v_{t,x})$ as the image vector $dF_{\mu_t}(v_t)_y$ of the vector field $v_t$ at point $y$, where $x$ stands for the elements of the fiber $f^{-1}(y)$. In case $f$ is injective, $dF_\mu(v)$ reduces to $df(v)$ for every $\mu$, which then can be regarded as full-fledged vector field.

Our naming of the vector field along $F(\mu_t)$, $dF_{\mu_t}(v_t)$ is, of course, very suggestive. Indeed, since the map $(v,\mu)\mapsto dF_{\mu}(v)$ is linear in $v$, Theorem $\ref{thm.ac}$ supports a natural definition for a notion of differentiability for absolutely continuous maps $F$. However, before we give such a definition, we need to make some further preparatory observations. 
Let us first continue with proving Theorem \ref{thm.ac}.

\begin{proof}
Let $\mu_t$ be an a.c. curve. Using Theorem \ref{them.ac}, we want to prove that there exists a family of vector fields $(\tilde{v}_t)_{t\in[0,1]}$ with $\int_0^1\|\tilde{v}_t\|_{L^2(F(\mu_t))}\ dt<\infty$, such that $(F(\mu_t),\tilde{v}_t)$ is an a.c. couple.

Let $(v_t)_{t\in[0,1]}$ be the tangent vector field of $\mu_t$. For each $t$ for which $v_t\in T_{\mu_t}W(M)$ (i.e. almost everywhere) we define $dF_{\mu_t}(v_t)$ as in equation $\eqref{eq.vf}$. We will prove that $dF_{\mu_t}(v_t)$ is an example of such vector fields $\tilde{v}_t$ we are looking for. 

Let us first see that $\int_0^1\|dF_{\mu_t}(v_t)\|_{L^2(F(\mu_t))}\ dt<\infty$. Using the triangle inequality for Bochner integrals, Jensen's inequality, the disintegration theorem and H\"older's inequality (in this order), we have:

\begin{eqnarray}\nonumber
& & \int_0^1\|dF_{\mu_t}(v_t)\|_{L^2(F(\mu_t))}\ dt  =  \int_0^1\sqrt{\int_N\|dF_{\mu_t}(v_t)\|^2_{T_yN}\ dF(\mu_t)(y)}\ dt%\\\nonumber
\end{eqnarray}
\begin{eqnarray}\nonumber
                          &=& \int_0^1\sqrt{\int_N \|\int_{f^{-1}(y)}df_x(v_{t,x})\ d\mu^y_t(x)\|^2_{T_yN}\ df_\#\mu_t(y)}\ dt\\\nonumber
                          &\leq& \int_0^1\sqrt{\int_N\left(\int_{f^{-1}(y)}\|df_x(v_x)\|_{T_yN}\ d\mu_t^y(x)\right)^2 df_\#\mu_t(y)}\ dt\\\nonumber
                          &\leq& \int_0^1\sqrt{\int_N\int_{f^{-1}(y)}\|df_x(v_{t,x})\|^2_{T_yN}\ d\mu^y_t(x)\ df_\#\mu_t(y)}\ dt\\\nonumber
                          &=& \int_0^1\sqrt{\int_M \|df_x(v_{t,x})\|^2_{T_{f(x)}M}\ d\mu_t(x)}\ dt\\\nonumber
                          &\leq& \int_0^1\sqrt{\int_M \norm{df_x}^2\cdot \|v_{t,x}\|^2_{T_xM}\ d\mu_t(x)}\ dt \\\nonumber
                          &\leq& \int_0^1\sqrt{\int_M \|v_{t,x}\|^2_{T_xM}\ d\mu_t\ \cdot\  \mathrm {ess} \sup\nolimits^{\mu_t} _{x\in M}\norm{df_x}^2}\ dt\\\nonumber
                          &=& \int_0^1\sqrt{\|v_t\|^2_{L^2(\mu_t)} \cdot\  \mathrm {ess} \sup\nolimits^{\mu_t} _{x\in M}\norm{df_x}^2}\ dt\\\nonumber
                          &\leq& C\int_0^1\|v_t\|_{L^2(\mu_t)}\ dt\ < \infty.
\end{eqnarray}

\noindent With $\mathrm {ess} \sup^{\mu_t} _{x\in M}$ we mean the essential supremum with respect to the measure $\mu_t$ and $C:=\mathrm {ess} \sup^{\mu_t} _{x\in M}\norm{df_x}^2$. The last expression is finite, since we know that $\|v_t\|_{L^2(\mu_t)}\leq|\dot{\mu_t}|$ for almost every $t$ and that the metric derivative of an a.c. map is integrable.
(The calculation above shows in particular that $dF_{\mu_t}(v_t)\in L^2(\mu_t)$ for almost every $t$, as we will point out again below.)
The disintegration theorem now allows the following calculation, with $g$ being the Riemannian tensor on $N$ and $h$ the one on $M$, $\varphi\in\mathcal{C}^{\infty}_c\left(N \times (0,1)\right)$ and $\nabla$ the gradient with respect to the first coordinate:

\begin{eqnarray}\nonumber
&   & \int_N{g_y\left(\nabla\varphi(y,t),dF_{\mu_t}(v_t)_y\right)\ df_\#\mu_t(y)}\\\nonumber 
& = &  \int_N{g_y\left(\nabla\varphi(y,t),\int_{f^{-1}({y})}df_x(v_{t,x})d\mu_t^y(x)\right)\ df_\#\mu_t(y)} \\\nonumber
& = & \int_N\int_{f^{-1}({y})}{g_y\left(\nabla\varphi(y,t),df_x(v_{t,x})\right)\ d\mu_t^y(x)df_\#\mu_t(y)}\\\nonumber
& = & \int_N\int_{f^{-1}({y})}{g_{f(x)}\left(\nabla\varphi(f(x),t),df_x(v_{t,x})\right)\ d\mu_t^y(x)df_\#\mu_t(y)}\\\nonumber
& = & \int_M {g_{f(x)}\left(\nabla\varphi(f(x),t),df_x(v_{t,x})\right)d\mu_t(x)}\\\nonumber
%& = & \int_M {g_{x}\left((df_x)^{ad}(\nabla\varphi)(f(x),t),v_{t,x}\right)d\mu_t(x)}\\\nonumber
& = & \int_M {h_x\left(\nabla(\varphi\circ f)(x,t),v_{t,x}\right)d\mu_t(x)}.
\end{eqnarray}

\noindent By $(\varphi\circ f)(x,t)$ we mean $(\varphi\circ (f\times id))(x,t)$.
For the second equality we used the continuity of the Riemannian tensor at every point $y\in N$. The last step is true because for every vector $X\in T_xM$,

\begin{eqnarray}\nonumber
h_x(\nabla(\varphi\circ f)(x),X) & = & X(\varphi\circ f)(x) =  df(X)(\varphi)(f(x))\\\nonumber
                                 & = & g_{f(x)}\left(\nabla\varphi(f(x)),d_xf(X)\right).
\end{eqnarray}

With this, we can now prove our claim that $\frac{d}{dt}F(\mu_t)+\nabla(dF_{\mu_t}(v_t) F(\mu_t))=0$ in the weak sense: For every $\varphi\in\mathcal{C}^{\infty}_c\left(N \times (0,1)\right)$ it is
\begin{eqnarray}\nonumber
&   & \int_{0}^{1}\int_N{\left(\frac{\partial}{\partial t}\varphi\right)(y,t)+g_y\left(\nabla\varphi(y,t),dF_{\mu_t}(v_t)_y\right)\ df_\#\mu_t(y)} dt \\\nonumber
& = & \int_{0}^{1}\int_M{\left(\frac{\partial}{\partial t}\varphi\right)(f(x),t)+ h_x\left(\nabla(\varphi\circ f)(x,t),v_{t,x}\right)}\ d\mu_t(x) dt\\\nonumber
& = & \int_{0}^{1}\int_M{\left(\frac{\partial}{\partial t}(\varphi\circ f)\right)(x,t)+ h_x\left(\nabla(\varphi\circ f)(x,t),v_{t,x}\right)}\ d\mu_t(x) dt.\\\nonumber
& = & 0.
\end{eqnarray}
Since $f$ is smooth and proper, $\varphi\circ f\in\mathcal{C}^{\infty}_c\left(M \times (0,1)\right)$ and we can apply our assumption on $(\mu_t,v_t)$ to be an a.c. couple.
\end{proof}
%
%---------------------------------------------------------
\subsection{About the image of $dF_\mu$}\label{sec.image}
%---------------------------------------------------------
%
For Theorem $\ref{thm.ac}$ we did not need to test whether $dF_{\mu}(v)\in T_{F(\mu)}W(N)$ for all $v\in T_\mu W(M)$, since we only needed $(F(\mu_t),dF_{\mu_t}(v_t))$ to be an a.c. couple. But is it still true, given that $(\mu_t,v_t)$ is a tangent couple? 

To begin with, the proof of Theorem \ref{thm.ac} also guarantees that for every $\mu\in W(M)$ and $v\in T_\mu W(M)$, $dF_{\mu}(v)\in L^2(F(\mu))$.
Knowing this, we can consider formula \eqref{eq.vf} as the prescription for a map between $T_\mu W(M)$ and $L^2(F(\mu))$. 

It is also useful to know that this map is always bounded, which we will see in the next proposition. For the rest of this section, let $F:W(M)\rightarrow W(N)$ be as in Theorem \ref{thm.ac} and $dF_{\mu}(v)$ as in formula \eqref{eq.vf}.

\begin{prop}[\emph{\textbf{Boundedness of $dF$}}]\label{thm.bd}
	For each $\mu\in W(M)$, $dF_\mu:T_\mu W(M)\rightarrow L^2(F(\mu))$ is bounded with 
	\begin{equation}\label{eq.bound}
	\norm{dF_\mu} \leq \mathrm {ess} \sup\nolimits _{x\in M}^\mu\norm{df_x}.
	\end{equation} Here, $\norm{\cdot}$ denotes the operator norm of the respective linear map and $\mathrm {ess} \sup _{x\in M}^\mu$ the essential supremum with respect to $\mu$. 
\end{prop}

Inequality \eqref{eq.bound} can be attained by taking similar steps as in the proof of Theorem \ref{thm.ac}. The right-hand side of equation \eqref{eq.bound} is finite since we demanded $\sup_{x\in M}\norm{dg_x}$ to be finite.

Let us give an example for a function $F$ for which equality is attained for every $\mu$ in inequality \eqref{eq.bound}.

\begin{example}\label{lem.preisom}
	Let $g:M\rightarrow M$ be a Riemannian isometry, i.e. $g^*h=h$, where $h$ is the Riemannian metric tensor on $M$. Then, for $F=g_\#$ and for all $\mu\in W(M)$, $\norm{dF_\mu} = \mathrm {ess} \sup _{x\in M}^\mu\norm{dg_x}=1$. 
    This is, because on the one hand, for all $x\in M$, $\norm{dg_x}=1$, since $dg$ is an isometry between the tangent spaces $T_xM$ and $T_{g(x)}M$. 
	On the other hand, 
	\begin{equation}\nonumber
	\norm{dg_\#} = \sup_{\|v\|_{T_\mu W(M)}=1}\|dg(v)\|_{T_{g_\#\mu}W(M)}=\sup_{\|v\|_{T_\mu W(M)}=1}\|v\|_{T_{\mu}W(M)}=1.
	\end{equation} 
\end{example}

To come back to our question, whether $dF_{\mu}(v)$ is always an element of $T_{F(\mu)}W(M)$, we first want to study the following simple cases.

\begin{lemma}\label{lem.extreme}
	Let $\mu=\delta_x$, for $x\in M$. Then $dF_{\mu}(v)\in T_{F(\mu)}W(N)$ for all $v\in T_\mu W(M)$.
\end{lemma}

\begin{proof}
	This is true because $F(\delta_x)=\delta_{f(x)}$ and for every $y\in N$, $L^2(\delta_y)\cong\R^n\cong T_{\delta_y}W(N)$, $n=dim\ N$. 
\end{proof}

\begin{lemma}\label{lem.isograd}
	Let $g:M\rightarrow M$ be a Riemannian isometry, i.e. $g^*h=h$, and $v=\nabla\varphi\in T_\mu W(M)$, $\varphi\in\mathcal{C}_c^\infty(M)$. Then for every $\mu\in W(M)$, $dF_\mu(v)=dg(v)=\nabla(\varphi\circ g^{-1})\in T_{F(\mu)}W(M)$.
\end{lemma}

\begin{proof}
	For the Riemannian metric $h$ on $M$ and for every vector field $X$
	\begin{eqnarray}\nonumber
	h(\nabla (\varphi\circ g^{-1}),X) &=& d(\varphi\circ g^{-1})(X)=d\varphi(dg^{-1}(X))=h(\nabla\varphi, dg^{-1}(X))\\\nonumber
	&=& h(dg(\nabla\varphi),X).
	\end{eqnarray}
\end{proof}

Since we know from Proposition \ref{thm.bd} that $dg_{\#\mu}$ is bounded and therefore continuous for every $\mu\in W(M)$, we can infer the following more general statement.

\begin{corollary}\label{cor.tang}
	Let $g:M\rightarrow M$ be a Riemannian isometry and $T_\mu W(M)\ni v=\lim_{n \rightarrow \infty}\nabla\varphi_n$. Then $dg(v)=\lim_{n \rightarrow \infty}\nabla(\varphi_n\circ g^{-1})\in T_{F(\mu)}W(M)$.
\end{corollary}

%\begin{example}
%	A simple example for an isometry on $\R^3$ is $f(x_1,x_2,x_3)=(x_1,x_3,x_2)$. For a vector field $v(x)=(v^1(x),v^2(x),v^3(x))$ it is then, for every $\mu\in W(\R^3)$, $dF_\mu(v)=dg(v)=(v^1(g(x)),v^3(g(x)),v^2(g(x)))$. Furthermore, we can compute directly that 
%	$\nabla(\varphi\circ g^{-1})=(\partial_{x_1}\varphi, \partial_{x_3}\varphi, \partial_{x_2}\varphi)=dg(\nabla\varphi)$.
%\end{example}

However, the case in Lemma \ref{lem.extreme} is extreme and the choice of functions in Lemma \ref{lem.isograd} specific.
We will now see that it can well be that $dF_\mu$ does not always hit the tangent space at $F(\mu)$.

\begin{thm}\label{thm.notin}
	Let $M$ be a compact manifold without boundary and $f=\id_M:(M,h_1)\rightarrow (M,h_2)$ the identity map on $M$, where $h_2=\nu^2 h_1$ and $\nu:M\rightarrow (0,\infty)$ nonconstant.
	Then for $F=\id_\#:W(M,h_1)\rightarrow W(M,h_2)$ there exists a $\nabla\varphi\in T_\mu W(M,h_1)$ so that $dF_\mu(\nabla\varphi)\notin T_\mu W(M,h_2)$, where $\mu=C\cdot\mu_{h_1}$, $\mu_{h_1}$ the volume measure on $M$ with respect to $h_1$ and $C=1/\mu_{h_1}(M)$.
\end{thm}

\begin{proof}
	It is clear that $F=\id_{W(M)}$ and $dF_\mu(v)=v\ \forall v\in T_\mu W(M,h_1)$. However, $v$ is not automatically a member of $T_\mu W(M,h_2)$. We will show that if $\varphi$ is chosen appropriately, $v=\nabla^{h_1}\varphi$ is not a limit of gradients with respect to $h_2$. 
	
	For this, recall that on a general Riemannian manifold $(M,h)$, there is a duality between vector fields $v$ and 1-forms $v^\flat$ by the formula $v_h^\flat(\cdot):=h(v,\cdot)$, which maps the vector field $\nabla^h\varphi$ to the $1$-form $d\varphi$. This identification gives an isomorphism between
	$\overline{\{\nabla^h\varphi\}}^{L^2(TM,h,\mu)}$ and
	$\overline{\{d\phi\}}^{L^2(T^*M,h^*,\mu)}$. Since this isomorphism depends on the chosen metric, it is in general $v^{\flat}_{h_1}\neq v^{\flat}_{h_2}$, but rather $v^{\flat}_{h_2}=\nu^2 v^{\flat}_{h_2}$, as Lemma \ref{lem.2} below shows. And thus $\nabla^{h_1}\varphi^{\flat}_{h_2}=\nu^2d\varphi$.
	
	Now $d(\nu^2d\varphi)=d(\nu^2) \wedge d\varphi$ which one can easily arrange to be
	non-zero. From Lemma \ref{lem.1} below we can thus infer that $\nabla^{h_1}\varphi^{\flat}_{h_2}\notin \overline{\{d\varphi\}}^{L^2(T^*M,h_2^*,\mu_{h_2})}$.
	As $C\mu_{h_1}=C\nu^n\mu_{h_2}$, with $n=\dim(M)$, the topology on $L^2(T^*M,h_2,C\mu_{h_1})$ and $L^2(T^*M,h_2,\mu_{h_2})$ coincide, so one can conclude that $\nu^2d\varphi$ is not an element of $T_{\mu}W(M,h_2)$.
\end{proof}

\begin{lemma}\label{lem.1}
	If $\omega$ is a smooth $1$-form on $M$ with $d\omega\ne 0$ then
	$\omega\notin \overline{\{d\varphi\}}^{L^2(T^*M,g^*,\mu_h)}$, where $\mu_h$ is the volume measure on $M$ with respect to h.
\end{lemma}
\begin{proof}
	Assuming the opposite and using the standard inner products, one gets the following contradiction:
	\begin{equation}
	0\ne    (d\omega,d\omega) = (\omega,d^*d\omega)	= \lim( d\varphi_n,d^*d\omega)=
	\lim (dd\varphi_n,d\omega)=\lim 0 =0
	\end{equation}
\end{proof}

\begin{lemma}\label{lem.2}
	In the situation of Theorem \ref{thm.notin} and interpreting $dF$ as a map of $L^2$-one forms, we have $dF_\mu(\omega) = \nu^2\omega$.
\end{lemma}
\begin{proof}
	Every vector field $v\in TM$ corresponds to the covector field $\omega\in T^*M$ by $\omega(w)=h_1(v,w)$. A change of the Riemannian metric $h_1$ to $h_2=\nu^2h_1$ yields $h_2(v,w)=\nu^2 h_1(v,w) =h_1(v,\nu^2w)=\omega(\nu^2 w)=\nu^2\omega(w)$, so with respect to $h_2$, $v$ corresponds to $\nu^2 \omega$.
\end{proof}
%
%-------------------------------------------------------
\subsection{Differentiable maps between Wasserstein\- spaces}
%-------------------------------------------------------
%
As we have seen in Subsection \ref{sec.image}, the conditions of Theorem \ref{thm.ac} do not guarantee $dF_{\mu}(v)\in T_{F(\mu)}W(N)$, even though this property is neccessary for a meaningful definition of the differential of $F$. To help us here, we use the fact that $L^2(\nu)=T_\nu W(N)\oplus T_\nu^\perp W(N)$ for every $\nu\in W(N)$ and compose $dF$ with a projection onto $T_{F(\mu)} W(N)$, so that at least $P^{F(\mu)}\circ dF_\mu: T_\mu W(M)\rightarrow T_{F(M)}W(M)$ is a linear and bounded map between $T_\mu W(M)$ and $T_{F(M)}W(M)$.

\begin{dfn}
	We call $P^{\mu}$ the orthogonal linear projection 
	\begin{eqnarray}\nonumber
	P^\mu:L^2(\mu) & \longrightarrow &  T_\mu W(M)\\\nonumber
	v\ \ \ \ & \longmapsto & \ \ \ \ v^\top,	
	\end{eqnarray}
	where $v=v^\top+v^\perp$, with $v^\top\in T_\mu W(M)$ and $v^\perp\in T_\mu^\perp W(M)$.
\end{dfn}

\begin{prop}\label{prop.proj}
 For every a.c. couple $(\mu_t,v_t)$, $(\mu_t,P^{\mu_t}(v_t))$ is a tangent couple.
\end{prop}

\begin{proof}
	Let $(\mu_t,v_t)$ be an a.c. couple, then, for $v_t=v_t^\top+v_t^\perp$ we have
	\begin{equation}\nonumber
	\frac{d}{dt}\mu_t+\nabla\cdot(v_t^\top\mu_t)=\frac{d}{dt}\mu_t+\nabla\cdot((v_t^\top+v_t^\perp)\mu_t)=0.
	\end{equation}
	And since $\|P^{\mu_t}(v_t)\|_{L^2(\mu_t)} \leq \|v_t\|_{L^2(\mu_t)}$ we have also $\|P^{\mu_t}(v_t)\|_{L^2(\mu_t)} \in L^1(0,1)$. Thus, $(\mu_t,P^{\mu_t}(v_t))$ is an a.c. couple and with Remark \ref{cor.tan} a tangent couple.
\end{proof}

%This means that, even if $dF_\mu$ does not hit $T_{F(\mu)}W(M)$, at least $P^{F(\mu)}\circ dF_\mu: T_\mu W(M)\rightarrow T_{F(M)}W(M)$ is a linear and bounded map between $T_\mu W(M)$ and $T_{F(M)}W(M)$ and for every tangent couple $(\mu_t,v_t)$, $\left(F(\mu_t),(P^{F(\mu_t)}\circ dF_{\mu_t})(v_t)\right)$ is a tangent couple, too.

With the observations we have collected so far, we can finally give our definition of a differentiable map between Wasserstein spaces.

\begin{dfn}[\emph{\textbf{Differentiable map between Wasserstein spaces}}]\label{dfn.diff}
An absolutely continuous map $F:W(M)\rightarrow W(N)$ is called \emph{differentiable} in case for every $\mu\in W(M)$ there exists a bounded linear map $dF_{\mu}:T_\mu W(M) \rightarrow T_{F(\mu)}W(N)$ such that for every tangent couple $(\mu_t,v_t)$ the image curve $dF_{\mu_t}(v_t)$ is a tangent vector field of $F(\mu_t)$. In this way a bundle map\footnote{In our sense of the word ``bundle''.} $dF:TW(M) \rightarrow TW(N)$ is defined which we want to call the \emph{differential} of F.
\end{dfn}
 When we say a map $F:W(M)\rightarrow W(N)$ is differentiable we automatically mean that it is absolutely continuous in the first place. 
\begin{remark}\label{rem.dfn}
	 The reader might be surprised that we only give a global definition of differentiability, without having started with a pointwise definition. The latter is difficult, if at all possible, since the tangent vector fields $v_t$ are only defined for a.e. $t\in[0,1]$, so a pointwise evaluation of these is not well-defined. The situation would change if one would be able to speak about \emph{continuous} curves of tangent vector fields, but it doesn't seem to be so easy to make this notion precise: For differing $t,t'$ the vector fields $v_t$ and $v_{t'}$ are elements of different tangent spaces, potentially even of different dimension, which is why the usual notion of continuity cannot be trivially applied.
\end{remark}
 Note again that $dF_{\mu_t}(v_t)$ is only well-defined almost everywhere, since $v_t$ is. But this is not harmful to our definition since in particular also the tangent vectors of $F(\mu_t)$ are only well-defined almost everywhere. But in this same manner, Definition \ref{dfn.diff} does not guarantee uniqueness of $dF$ in a strict sense. (Here we mean that $dF=\widetilde{dF}$ whenever $dF_\mu(v)=\widetilde{dF}_\mu(v)$ for all $(\mu,v) \in TW(M)$.) But, after all, one can say that $dF$ is unique up to a ``negligible'' set. 
\begin{dfn}[\emph{\textbf{Negligible set}}]\label{dfn.neg}
	A subset $Z\subset TW(M)$ is called \emph{negligible} whenever for every tangent couple $(\mu_t,v_t)$ the set $\{t\in (0,1)\mid (\mu_t,v_t)\in Z\}$ is of Lebesgue measure zero.
\end{dfn}
	\noindent This definition respects the $L^1(dt)$-nature of the $v_t$'s in the sense that changing any $v_t$ on a set of measure zero does not change the measure of the set $\{t\in (0,1)\mid (\mu_t,v_t)\in Z\}$.
\begin{prop}[\emph{\textbf{Uniqueness of the differential}}]
The differential $dF$ of a differentiable map $F:W(M)\rightarrow W(N)$ is unique up to a redefinition on a negligible set $Z\subset TW(M)$.
\end{prop}
\begin{proof}
	Let $dF$ and $\widetilde{dF}$ be two pointwise linear bundle maps, $dF$ being the differential of an a.c. map $F$. It is to show that $dF$ and $\widetilde{dF}$ are both a differential of $F$ if and only if $\{(\mu,v)\in TW(M)\mid dF_\mu(v)\neq\widetilde{dF}_\mu(v)\}$ is negligible.\\
	Let $dF$ and $\widetilde{dF}$ be different only on a negligible set. In this case, for each tangent couple $(\mu_t,v_t)$ the image velocities $\widetilde{dF}_{\mu_t}(v_t)$ are different from the ones of $dF_{\mu_t}(v_t)$ only on a null set and thus still equal the tangent vector fields along $F(\mu_t)$ almost everywhere.
	Let on the other hand $dF$ and $\widetilde{dF}$ both fulfill the conditions of Definition \ref{dfn.diff}. By definition, for each tangent couple $(\mu_t,v_t)$ both $dF_{\mu_t}(v_t)$ and $\widetilde{dF}_{\mu_t}(v_t)$ are equal almost everywhere to the tangent vectors along $F(\mu_t)$. Thus, for every tangent couple $(\mu_t,v_t)$, $\{t\in(0,1)\mid dF_{\mu_t}(v_t)\neq\widetilde{dF}_{\mu_t}(v_t)\}$ has Lebesgue measure zero.
\end{proof} 	
%
%family of linear maps $dF_{\mu}:T_\mu W(M) \rightarrow T_{F(\mu)}W(N)$ such that for every absolutely continuous couple $(\mu_t,v_t)$ with $\mu_0=\mu$ and $v_0=v$ the tangent vector field $\tilde{v}_t$ of the curve $F(\mu_t)$ is such that $\tilde{v}_0=dF(\tilde{v}_0)$ whenever $\tilde{v}_0$ is defined.
%formed in the following way. Let $v\in T_{\mu}W(M)$ and $(\mu_t, v_t), t\in (-\varepsilon,\varepsilon)$, any [pair] such that $\mu_0=\mu$, $v_0=v$. Then we define $dF(v)=w_0$, where $w_t\in L^2(F(\mu_t),N)$ is the unique minimal vector field satisfying the continuity equation together with $(F(\mu_t))_{t\in(-\varepsilon,\varepsilon)}$.
%
Let us now analyse some properties of negligible sets.
\begin{prop}\label{prop.prop}
	\begin{itemize}
		\item[1.)] $T_\mu(W(M))\setminus\{0\}$ is negligible, for every $\mu\in W(M)$. But $T_\mu(W(M))$ isn't. 
		\item[2.)] The countable union of negligible sets is negligible.
		\item[3.)] Every subset of a negligible set is negligible.
		\item[4.)] The following is an equivalence relation on the set of mappings between tangent bundles on Wasserstein spaces:\\  $F\sim G:\Leftrightarrow \{(\mu,v)\in TW(M)\mid F(\mu,v)\neq G(\mu,v)\}$ is negligible.
	\end{itemize}
\end{prop}
\begin{remark}
Let $dF$ be a differential of a map $F:W(M)\rightarrow W(N)$. Then there are members of its equivalence class $[dF]$ which are not a differential of $F$ since not every member has to be pointwise linear and bounded. Restricting, however, the equivalence relation onto the subset of pointwise linear and bounded maps between tangent bundles of Wasserstein spaces solves this issue. In this case $[dF]$ contains precisely all the possible differentials of $F$. Whenever we refer to a representative of $dF$, we mean an element of the latter equivalence class.
\end{remark}
\begin{proof}
	\begin{itemize}
		\item [1.)] Let $(\mu_t,v_t)$ be a tangent couple, $v_t$ a fixed representative of $v_t\in L^1(dt)$ and $T_\mu:=\{t\in(0,1)\mid \mu_t=\mu,\ v_t\in T_\mu W(M)\}$ for some $\mu\in W(M)$. Let us further assume that $v_t\neq0$ for every $t\in T_\mu$ which in particular means that $|\dot{\mu}_t|\neq0$ for every $t\in T_\mu$. From this we can also infer that for no $t_0\in T_\mu$ there exists a neighborhood on which $\mu_t$ is constant. Let $a\in T_\mu$ be a point which is not isolated. This means that in every neighborhood of $a$ is another point of $T_\mu$. The consequence of this would be that the metric derivative would not exist at that point which we excluded in the definition of $T_\mu$.
		%This can only be the case for a subset of $T_\mu$ of zero measure. All the other points of $T_\mu$ must consequently be isolated of which there can only be countably many. \\
		So $T_\mu$ must consist of only isolated points and thus must be countable. Choosing another representative of $v_t\in L^1(\mu)$ only changes the amount of $t$'s in $T_\mu$ by a null set.\\
		$T_\mu(W(M))$  is not negligible since $\mu_t=\mu$ is absolutely continuous with metric derivative $0$.
		\item [2.)] This follows from the fact that any countable union of sets of measure zero again is of measure zero.
		\item [3.)] Let $N$ be a subset of a negligible set and $(\mu_t,v_t)$ an a.c. curve with a fixed representative $v_t$. The amount of times where $(\mu_t,v_t)\in N$ can only be a subset of a set of zero measure. Since the Lebesgue measure is a complete measure this subset itself is measurable and in particular of measure zero.		
		\item[4.)] This follows from 1.) and 2.)
	\end{itemize}
\end{proof}
The following corollary finally recovers the properties expected of a differential.
\begin{corollary}\label{cor.prop}
\begin{itemize}
\item[1.) ] In case $F=f_\#$ and $f$ is as in Theorem \ref{thm.ac}, $F$ is differentiable with $dF_\mu=P^{F(\mu)}\circ \widehat{dF}_\mu$, where 
$P^{F(\mu)}$ is the orthogonal projection onto $T_{F(\mu)}N$ from Proposition \ref{prop.proj} and 
\begin{equation}\nonumber
\widehat{dF}_{\mu}(v)_y:= \int_{f^{-1}({y})}df(v_{x})d\mu^y(x), 
\end{equation}
as in formula \eqref{eq.vf}. In case $f$ is a Riemannian isometry, the additional projection $P$ is not necessary, as we have seen in Corollary \ref{cor.tang}. Then, $dF_\mu=df$ for all $\mu\in W(M)$.
\item[2.) ] In particular, the identity mapping $F(\mu)=\mu$ is differentiable with $dF_\mu(v)=v$ up to a negligible map.
\item[3.) ] Let $F:W(M)\rightarrow W(N)$ and $G:W(N)\rightarrow W(O)$ be two differentiable maps. Then also $G\circ F:W(M)\rightarrow W(O)$ is differentiable with $d(G\circ F)_\mu(v)=\left(dG_{F(\mu)}\circ dF_\mu\right)(v)$ up to a negligible set.
\item[4.) ] Whenever $F$ is differentiable, bijective with differentiable inverse $F^{-1}$, then $dF$ is also invertible with inverse $d(F^{-1})$, up to a negligible set. 
\end{itemize}
\end{corollary}
\begin{proof}
\begin{itemize}
\item[1.)] This follows from Theorem \ref{thm.ac} and Proposition \ref{prop.proj}.
\item[2.)] This is immediate.
\item[3.)] First we observe that the composition of two absolutely continuous maps between Wasserstein spaces is again absolutely continuous. Also, the composition of two bounded linear maps is again a bounded linear map. To show differentiability, we will check that $dG_{F(\mu)}\circ dF_\mu:T_\mu W(M)\rightarrow T_{(G\circ F)(\mu)}W(O)$ is such that for every tangent couple $(\mu_t,v_t)$, also $((G\circ F)(\mu_t),(dG_{F(\mu)}\circ dF_\mu)(v_t))$ is a tangent couple. So let $(\mu_t,v_t)$ be a tangent couple. Since $F$ is differentiable, we know that $(F(\mu_t),dF_{\mu_t}(v_t))$ is a tangent couple. Similarly, also $\left(G(F(\mu_t)),dG_{F(\mu_t)}(dF_{\mu_t}(v_t))\right)$ is a tangent couple. Since $G(F(\mu_t))=(G\circ F)(\mu_t)$ and $dG_{F(\mu_t)}(dF_{\mu_t}(v_t))=(dG_{F(\mu_t)}\circ dF_{\mu_t})(v_t)$, we have proven the claim.
\item[4.)] This is an immediate consequence of 2.) and 3.).
\end{itemize}
\end{proof}
\begin{remark}
		Let us again emphasize that this type of differentiability is highly tailored to the structure given by optimal transport. It knowingly does not fit into the framework of, e.g., \cite{Kriegl_1997}. Nevertheless, let us mention that also in this reference, the notion of differentiable maps between infinite dimensional manifolds is established via the property that differentiable curves should be mapped to differentiable curves. 
\end{remark}

\subsection{Pullbacks and formal Riemannian isometries}

As an application of the previous section, we propose a definition for the pullback of the formal Riemannian tensor on $W_2(M)$ and furthermore a definition for formal Riemannian isometries. As the formal Riemannian metric was defined by comparison of formulae to actual Riemannian structures (see Definition \ref{df.Riem}), the performance of pullbacks now gives rise to definitions of further possible formal Riemannian metrics on $W_2$-spaces, in cases where $dF_\mu$ is injective for every $\mu$, i.e. in case $F$ can be considered to be an immersion. 
\begin{dfn}[\emph{\textbf{Pullback of the formal Riemannian tensor}}]\label{df.pb}
	Let $F:W(N)\rightarrow W(M)$ be differentiable, $dF$ be a fixed differential of $F$, $\mu\in W(N)$ and $H_{F(\mu)}$ the formal Riemannian metric tensor on $W(M)$ at point $F(\mu)\in W(M)$. Then, for $v,w\in T_\mu W(M)$, the \emph{pullback} $(F^*H)_\mu$ of $H_{F(\mu)}$ is defined as $$(F^*H)_\mu(v,w):=H_{F(\mu)}(dF_\mu(v),dF_\mu(w)).$$
\end{dfn}
Unfortunately, this definition depends on the choice of the differential of $F$, which is, as we have seen, only unique up to a negligible set. 
\begin{dfn}[\emph{\textbf{Formal Riemannian isometry}}]	\label{df.riso}
	Analogously to the finite dimensional case, we call a bijective differentiable map $F:W(M)\rightarrow W(M)$ with differentiable inverse a \emph{formal Riemannian isometry}, in case there is a representative of $dF$ such that for all $\mu\in W(M)$ $(F^*H)_\mu(v,w)=H_\mu(v,w)$ for all $(v,w)\in T_\mu W(M)\times T_\mu W(M)$.	
\end{dfn}
%We want to call such a representative as in Definition \ref{df.riso} \emph{suitable}.
It is straightforward to see that $F$ is a formal Riemannian isometry iff there is a representative of $dF$ such that for every $\mu\in W(M)$ $dF_\mu:T_\mu W(M)\rightarrow T_{F(\mu)} W(M)$ is a metric isometry with respect to the metrics induced by the $L^2$-norms.

Important formal Riemannian isometries are generated by the isometry group of the underlying metric space. By means of the pushforward, $ISO(M)$ acts isometrically also on $\PM_p$ and the map
	\begin{eqnarray*}
	G\times TW(M) &\rightarrow& TW(M)\\
	\left(g,(\mu,v)\right) &\mapsto& (g_\#\mu,dg(v))
    \end{eqnarray*}
defines an induced action of every subgroup $G$ of $ISO(M)$ on the tangent bundle of $W(M)$, where we regard $dg$ as a differential of $g_\#$. It is quick to check that for $g\in ISO(M)$, $g_\#:W(M)\rightarrow W(M)$ is a formal Riemannian isometry.

\begin{lemma}\label{lem.giso}
	Let $g\in ISO(M)$, then $T_{g_\#\mu} W(M)=dg\left(T_\mu W(M)\right)$ for all $\mu\in W(M)$. Here, we again regard $dg$ as a, fixed, differential of $g_\#$.
\end{lemma}
%
%Of course, this lemma holds true for every bijective differentiable map $F$ whose inverse map is differentiable, too.

\begin{prop}\label{prop.infriem}
	Every formal Riemannian isometry is an isometry in the metric sense of its Wasserstein space. 	
\end{prop}
\begin{proof}
	Let $F$ be a formal Riemannian isometry. Since by definition $F$ is bijective with differentiable inverse, every a.c. couple $(\mu_t,v_t)$ can be represented as the image of another a.c. couple $(\tilde{\mu}_t,\tilde{v}_t)$. Just choose $\tilde{\mu}_t:=F^{-1}(\mu_t)$ and $\tilde{v}_t:=dF^{-1}(v_t)$. Then, $\mu_t=F(\tilde{\mu}_t)$ and, using Corollary \ref{cor.prop}, $v_t=dF(\tilde{v}_t)$ almost everywhere.
	Conversely, every image of an a.c. couple, in the above sense, is an a.c. couple. Let $dF$ be a suitable representative.
	For $\mu,\nu\in W(M)$ and $\mu_t$ a.c. connecting them, we then have according to \ref{thm.bb}:
	\begin{eqnarray}\nonumber
	W(F(\mu),F(\nu)) &=& \inf_{\left(F(\mu_t),dF(v_t)\right)}\int_0^1 \sqrt{H_{F(\mu_t)}(dF(v_t),dF(v_t))}\ dt\\ \nonumber
	&=& \inf_{(\mu_t,v_t)}\int_0^1\sqrt{H_{\mu_t}(v_t,v_t)}\ dt\ =\  W(\mu,\nu).
	\end{eqnarray}
\end{proof}
It would be interesting to find out whether the converse implication of Proposition \ref{prop.infriem} is true as well, as it is the case for finite dimensional Riemannian manifolds.
%
%%%%%%%%%%%%%%%%%%%%%%%%%%%%%%%
\subsection{Convex mixing of maps}
%%%%%%%%%%%%%%%%%%%%%%%%%%%%%%%
%
In the examples, we so far have only been concerned with maps $F:W(M)\rightarrow W(N)$ which are induced by maps $f:M\rightarrow N$. Now one could wonder how a map $F$ which is not of this type could look like and what its differentiability properties are. As a first hint, we recall that whenever there is an $f:M\rightarrow N$ such that $F=f_\#$, then for $x\in M$ it is $F(\delta_x)=\delta_{f(x)}$. Based on this, we can construct the following examples.
\begin{example}\label{ex.push}
	\begin{itemize}
		\item If $F(\mu)=\mu_0$ is a constant map such that $\mu_0\neq \delta_{y_0},\ y_0\in N$, then there exists no map $f:M\rightarrow N$ such that $F=f_\#$. 
		In case $F(\mu)=\delta_{y_0}$, it is $F=f_\#$ with $f(x)=y_0\ \forall x\in M$.
		\item Let $F_{i}:W(M)\rightarrow W(N)$, $i=1,2$, such that they do not coincide on $\{\delta_x\mid x\in M\}$. The mixing of measures $F:=(1-\lambda)F_1+\lambda F_2$ for $0<\lambda<1$, then, cannot be a pushforward of measures.
	\end{itemize}	
\end{example}
\begin{remark}
Another way to think about this issue is the following: Every map $F:W(M)\rightarrow W(N)$ has a decomposition into a map $\tilde{F}:W(M)\rightarrow \PM(M\times N)$ with $\pi^1_\#\tilde{F}(\mu)=\mu$ and the map $\pi^2_\#:\PM(M\times N)\rightarrow W(N)$, i.e. $F=\pi^2_\#\circ\tilde{F}$. Certainly, $\tilde{F}$ is not unique, but one can always choose $\tilde{F}(\mu)=\mu\otimes F(\mu)$. Thus, $F$ is a pushforward with respect to a map $f$ if and only if there exists a map $\tilde{F}$ in such a way that $\tilde{F}(\mu)=(Id,f)_\#\mu$. According to \cite{Ambrosio2013}, Lemma 1.20 this is equivalent to saying that for every $\mu$ there exists a $\tilde{F}(\mu)$-measurable set $\Gamma\subset M\times N$ on which $\tilde{F}(\mu)$ is concentrated such that for $\mu$-a.e. $x$ there exists only one $y=f(x)\in M$ with $(x,y)\in\Gamma$. And in this case, $\tilde{F}(\mu)=(Id,f)_\#\mu$. 
\end{remark}
It is easy to see that any constant map $F:W(M)\rightarrow W(N),\ \mu\mapsto \mu_0$, is differentiable with $dF=0$ up to a negligible set. In the following we will investigate whether maps of the form $F=(1-\lambda)\ F_1+\lambda\ F_2$ are also differentiable. Let us start with asserting that the convex mixing of of a.c. maps is a.c..
\begin{prop}\label{lem.Fac}
	Let $F_i:W(M)\rightarrow W(N)$, $i=1,2$, be arbitrary a.c. maps. Then, for $0\leq\lambda\leq 1$, also $F:=(1-\lambda)\ F_1+\lambda\ F_2$	 is a.c.
\end{prop}
For the proof of Proposition \ref{lem.Fac} we will use that already the convex mixing of of a.c. curves is a.c.
\begin{lemma}\label{prop.conv}
	Let $\mu_t^1$ and $\mu_t^2$ be a.c. curves. Then also the convex mixing $\mu_t:=(1-\lambda)\mu_t^1+\lambda\mu_t^2$ with $0\leq\lambda\leq 1$ is an a.c. curve.
\end{lemma}
\begin{proof}
	Since the $\mu^i_t$ are a.c. curves, for every $s\leq t\in(0,1)$
	there is a $g_i\in L^1(0,1)$ such that $$W\left(\mu^i_s,\mu^i_t\right)\leq \int_s^t g_i(\tau)\ d\tau.$$  
	Now let $\gamma_i\in Adm(\mu^i_s,\mu^i_t)$. Then $(1-\lambda)\gamma_1+\lambda\gamma_2\in Adm\left(\mu_s,\mu_t\right).$
	This is because for every measurable set $A$ and $\pi^i$ the projection onto the $i$-th component,
	\begin{eqnarray}\nonumber
	\pi^1_\#\left((1-\lambda)\gamma_1+\lambda\gamma_2\right)(A) & = & \left((1-\lambda)\gamma_1+\lambda\gamma_2\right)((\pi^1)^{-1}(A))\\\nonumber & = & (1-\lambda)\gamma_1((\pi^1)^{-1}(A))+\lambda\gamma_2((\pi^1)^{-1}(A))\\\nonumber
	& = & \left((1-\lambda)\mu^1_s+\lambda \mu^2_s\right)(A)\ =\ \mu_s(A).
	\end{eqnarray}
	Similarly for $\pi^2$. Then for $\widetilde{Adm}\left(\mu_s,\mu_t\right):=\{(1-\lambda)\gamma_1+\lambda\gamma_2\mid \gamma_i\in Adm(\mu^i_s,\mu^i_t)\}\subset Adm\left(\mu_s,\mu_t\right)$ we have
	\begin{eqnarray}\nonumber
	W(\mu_s,\mu_t)^2 & = & W\left((1-\lambda)\mu^1_s+\lambda \mu^2_s,(1-\lambda)\mu^1_t+\lambda \mu^2_t\right)^2\\\nonumber
	& \leq & \inf_{\pi\in\widetilde{Adm}(\mu_s,\mu_t)}\int d^2(x,y)\ d\pi(x,y) \\\nonumber
	& = & (1-\lambda)\ \inf_{\gamma_1\in Adm(\mu_s^1,\mu_t^1)}\int d^2(x,y)\ d\gamma_1 +\ \lambda\ \inf_{\gamma_2\in Adm(\mu_s^2,\mu_t^2)}\int d^2(x,y)\ d\gamma_2\\\nonumber
	& = & (1-\lambda)\ W(\mu^1_s,\mu^1_t))^2+ \lambda\ W(\mu^2_s,\mu^2_t)^2
	\end{eqnarray}
	This means that
	\begin{eqnarray}\nonumber
	W(\mu_s,\mu_t) & = & \sqrt{(1-\lambda)\ W(\mu^1_s,\mu^1_t))^2+ \lambda\ W(\mu^2_s,\mu^2_t)^2}\\\nonumber
	& \leq & \sqrt{(1-\lambda)}\ W(\mu^1_s,\mu^1_t)+ \sqrt{\lambda}\ W(\mu^2_s,\mu^2_t)\\\nonumber
	& \leq & \sqrt{(1-\lambda)} \int_s^t g_1(\tau)\ d\tau + \sqrt{\lambda}\int_s^t g_2(\tau)\ d\tau \\\nonumber
	& = & \int_s^t(\sqrt{(1-\lambda)}\ g_1+\sqrt{\lambda}\ g_2)\ d\tau.
	\end{eqnarray}
\end{proof}
Before continuing with the proof of Proposition \ref{lem.Fac} we give this immediate corollary from the proof of Lemma \ref{prop.conv}.
\begin{corollary} Let $(X,d)$ be a metric space and $\mu_{11},\mu_{12},\mu_{21},\mu_{22}$ four probability measures on $X$. Then, 
	$$W_p\left((1-\lambda) \mu_{11}+\lambda \mu_{12},(1-\lambda) \mu_{21}+\lambda \mu_{22}\right)\leq	\  \sqrt[p]{(1-\lambda)} W_p(\mu_{11},\mu_{21})+\sqrt[p]{\lambda} W_p(\mu_{12},\mu_{22}).$$
	%	\begin{eqnarray}\nonumber
	%	&  & W_p\left((1-\lambda)\ \mu^1_s+\lambda\ \mu^2_s,(1-\lambda)\ \mu^1_t+\lambda\ \mu^2_t\right)\\\nonumber
	%	& \leq & \sqrt[p]{(1-\lambda)}\ W_p(\mu^1_s,\mu^1_t)+\sqrt[p]{\lambda}\ W_p(\mu^2_s,\mu^2_t).
	%	\end{eqnarray}\qed	
\end{corollary}
\begin{proof}[Proof of Proposition \ref{lem.Fac}]
	Let $\mu_t$ be an a.c. curve. Then by definition $F_i(\mu_t)$, $i=1,2$, are a.c. curves. From Lema \ref{prop.conv} we now know that also $F(\mu_t)$ is an a.c. curve.
\end{proof}
\begin{thm}\label{thm.convex}
	Let $F_i:W(M)\rightarrow W(N)$, $i=1,2$, be two differentiable maps. Then $F=(1-\lambda)\ F_1+ \lambda\ F_2$ is differentiable.
\end{thm}
Since we have already seen that with the conditions of Theorem \ref{thm.convex} $F$ is a.c., as both $F_i$ are a.c., we know that $F$ maps a.c. curves to a.c. curves. We know further that along each of these a.c. image curves there has to be a tangent vector field. To find the tangent map, mapping curves of tangent vector fields along a.c. curves to the corresponding curves of tangent vector fields along the image a.c. curves, i.e. to prove the theorem, we first give a formula for a canonical image tangent vector field.
\begin{lemma}\label{lem.canonical}
	Let $F_i:W(M)\rightarrow W(N)$, $i=1,2$, be two differentiable maps. For an a.c. curve  $\gamma_t$ in $W(M)$, we define the a.c. curves $\mu_t:=F_1(\gamma_t)$, $\nu_t:=F_2(\gamma_t)$ and $\alpha_t:=\lambda \mu_t+(1-\lambda) \nu_t$ in $W(N)$. With the Lebesgue decompositiom theorem, the measures $\mu_t$ and $\nu_t$ give rise to unique	measures $\tau^\mu_t,\tau^\nu_t,\beta_t$ and Radon-Nykodym derivatives $\rho_t$ such that
	\begin{enumerate}
		\item For each $t$ the measures $\tau^\mu_t,\tau^\nu_t$ and $\beta_t$ are mutually
		singular: there exist Borel subsets $A_t,B_t,C_t$ that are pairwise disjoint with
		union $N$ such that $B_t$ and $C_t$ are nullsets for $\tau^\mu_t$, $A_t$ and
		$C_t$ are nullsets for $\tau^\nu_t$ and $A_t,B_t$ are nullsets for $\beta_t$.
		\item $\mu_t = \tau_t^\mu + \beta_t$
		\item $\nu_t=\tau_t^\nu + \rho_t \beta_t$
		\item $\rho_t$ is zero only on a nullset of $C_t$.
	\end{enumerate}
	If furthermore $v_t$ is a tangent vector field for $\mu_t$ and $w_t$ is an accompanying vector field for $\nu_t$, we can give the formula for a canonical accompanying vector field $u_t\in L^2(N,\alpha_t)$ for $\alpha_t$ as
	\begin{equation*}
	u_t(x):=
	\begin{cases}
	v_t(x); & x\in A_t\\
	w_t(x); & x\in B_t\\
	\frac{\lambda v_t(x)+ \rho_t(1-\lambda)w_t(x)}{\lambda+(1-\lambda)\rho_t}; & x\in C_t.
	\end{cases}
	\end{equation*}
\end{lemma}
\begin{proof}
	Since $\frac{d}{dt}\alpha_t$ is linear in $\alpha_t$, the
	continuity equation for $(\alpha_t,u_t)$ is satisfied if and only if
	\begin{equation}\label{eq:test}
	\int_0^T \int_N h(\nabla \phi(x,t),u_t(x)) d\alpha_t\,dt = \int_0^T (\int_N
	h(\nabla\phi(x,t),v_t(x)) \lambda d\mu_t + h(\nabla\phi(x,t),w_t(x))(1-\lambda) d\nu_t)\,dt
	\end{equation}
	for all $\varphi\in C_c^\infty\left((0,T)\times N\right)$ and $u_t\in L^2(N,\alpha_t)$.
	
	Let us first check that $u_t\in L^2(TN,\alpha_t)$. Since $N=A_t\dot\cup\, B_t\dot\cup\, C_t$, the condition can be checked separately on $A_t, B_t$ and $C_t$.
	First, 
	\begin{equation*}
	\int_{A_t} |u_t(x)|^2 d\alpha_t = \int_{A_t} |v_t(x)|^2 \lambda d\mu_t <\infty,
	\end{equation*}
	and similarly for $B_t$. To check the situation on $C_t$, we start with 
	\begin{eqnarray}\label{eq:L2}
	\int_{C_t} |u_t(x)|^2 d\alpha_t & = &\int_{C_t} \frac{ |\lambda v_t(x) + (1-\lambda)\rho_t
		w_t(x)|^2}{(\lambda + (1-\lambda)\rho_t)^2} (\lambda d\beta_t + (1-\lambda)\rho_t
	d\beta_t)\\ \label{eq:L22}
	 & \le&  2\int_{C_t} \left(\frac{\lambda}{\lambda+ (1-\lambda)\rho_t} |v_t(x)|^2 \lambda +
	\frac{(1-\lambda)\rho_t}{\lambda+(1-\lambda)\rho_t} |w_t(x)|^2 (1-\lambda)\rho_t \,\right)d\beta_t.
	\end{eqnarray}
	Now it holds that $\frac{\lambda}{\lambda+(1-\lambda)\rho_t}\le 1$ and
	$\int_{C_t} |v_t(x)|^2 d\beta_t <\infty$ (as one summand in the $L^2$-norm of $v_t$
	with respect to $\mu_t$). Similarly for the second summand, so we see that the
	whole expression in Equation \eqref{eq:L22} is finite.
	
	Let us now check Equation \eqref{eq:test}. This can be done separately for (almost all) $t\in [0,T]$ and again separately for the integrals over $A_t,B_t,C_t$.
	On $A_t$, Equation \eqref{eq:test} holds because here $u_t=v_t$ and $\alpha_t=\lambda
	\mu_t=\lambda \tau^\mu_t$, whereas $\nu_t(A_t)=0$. A similar argument works on $B_t$. 
	On $C_t$, formally,
	\begin{equation*}
	u_t d\alpha_t = \frac{\lambda v_t +(1-\lambda)\rho_t w_t}{\lambda +(1-\lambda)
		\rho_t} d(\lambda \beta_t+ (1-\lambda)\rho_t \beta_t) = \left(\lambda v_t +
	(1-\lambda)\rho_t w_t\right) d\beta_t = v_t \lambda d\mu_t + w_t (1-\lambda )
	d\nu_t.
	\end{equation*}
\end{proof}
\begin{proof}[Proof of Theorem \ref{thm.convex}]
	First, we need to check that $u_t$ is indeed an accompanying vector field for $\alpha_t$, i.e. that $\|u_t\|_{L^2(\alpha_t)}\in L^1(0,1)$, so that its projection onto the tangent spaces is indeed a tangent vector field along $\alpha_t$.
	
	Since $N=A_t\dot\cup\, B_t\dot\cup\, C_t$, 
	\begin{eqnarray}\nonumber
	\|u_t\|_{L^2(\alpha_t)} &=&\|\left.u_t\right|_{A_t}+ \left.u_t\right|_{B_t}+ \left.u_t\right|_{C_t}\|_{L^2(\alpha_t)}\leq \|\left.u_t\right|_{A_t}\|_{L^2(\alpha_t)}+\|\left.u_t\right|_{B_t}\|_{L^2(\alpha_t)}+\|\left.u_t\right|_{C_t}\|_{L^2(\alpha_t)} \\ \label{eq.10}
	&\leq& \sqrt{\lambda}\ \|v_t\|_{L^2(\mu_t)} + \sqrt{(1-\lambda)}\ \|w_t\|_{L^2(\nu_t)} + \|\left.u_t\right|_{C_t}\|_{L^2(\alpha_t)}.
	\end{eqnarray}
	We know of the first two summands in Equation \eqref{eq.10} that their $L^1(0,1)$-norm is finite, as we demanded $v_t$ and $w_t$ to be accompanying vector fields. It thus suffices to show the finiteness of the $L^1(0,1)$-norm of the last summand. Here, we find with $\bar{\rho}_{t,\lambda}:=\frac{1}{\lambda+(1-\lambda)\rho_t}$,
	$$\|\left.u_t\right|_{C_t}\|_{L^2(\alpha_t)}=\|\left.(\lambda v_t+(1-\lambda)\rho_t w_t)\right|_{C_t}\|_{L^2(\bar{\rho}_{t,\lambda} d\beta_t)} \leq \|\lambda \left.v_t\right|_{C_t}\|_{L^2(\bar{\rho}_{t,\lambda} d\beta_t)} + \|(1-\lambda)\rho_t \left.w_t\right|_{C_t}\|_{L^2(\bar{\rho}_{t,\lambda} d\beta_t)}.$$
	We have encountered both of those last summands in the proof Lemma \ref{lem.canonical} and analogously to there (where we have concluded the finiteness of the $L^2$-norm), we can now conclude the finiteness of the $L^1(0,1)$-norm of these summands and thus the claim that $\|u_t\|_{L^2(\alpha_t)}\in L^1(0,1)$.
	
	Finally, observe that the construction of $u_t$ from $(v_t,w_t)$ is a linear
	and bounded map $A_\lambda\colon L^2(M,\mu_t)\oplus L^2(M,\nu_t)\to
	L^2(M,\alpha_t)$, as the formula in the proof of the $L^2$-property of $u_t$ shows. Composition of $A_\lambda$ with $dF\oplus dG$ and the projection to the tangent space then defines the derivative of $\lambda F+(1-\lambda) G$ and shows that this convex combination is differentiable.
\end{proof}
%
%\begin{remark}
%	The preceding examination is one way to see that the notion of differentiable maps that we have built on top of the formal differentiable structure on $W(M)$ is not compatible with the convex structure of $W(M)$ other than in special cases.
%	(The map $F=(1-\lambda)\ F_1+\lambda\ F_2$ is still absolutely continuous, but this is not clear for maps $F$ with the property $F(\lambda\mu+(1-\lambda)\nu)=\lambda F(\mu)+(1-\lambda)F(\nu)$.) This is different from considering $\PM(M)$ as a convex subspace of the vector space of signed measures, where one can apply convenient calculus \cite{Kriegl_1997}. So pushing the formal calculus further in the manner that we have, one can see how it deviates from what we assume calculus has to provide, even though it matches some other formulas pretty well. 
%\end{remark} 

%%% Local Variables:
%%% mode: latex
%%% TeX-master: main
%%% End:

\appendix
%ANHANG
%
%------------------------------------
\section{Disintegration theorem}\label{A.dis}
%------------------------------------
%
To be able to prove Theorem \ref{thm.ac}, we rely on the following statement (see \cite{GFlows}).
\begin{thm}
Let $X$ and $Y$ be Radon spaces. Furthermore let $\mu\in\PM(X)$ and $f:X\rightarrow Y$ be a measurable map.
Then there exists a $f_\#\mu$-almost everywhere uniquely determined family of probability measures $\{\mu_y\}_{y\in Y}$ on $X$ such that 
\begin{itemize}
	\item for every measurable set $A\subset X$ the map $y\mapsto \mu_y(A)$ is measurable,
	\item $\mu_y(X\setminus f^{-1}(y))=0$ for $f_\#\mu$-almost every $y\in Y$,
	\item for every measurable function $g:X\rightarrow [0,\infty]$ it is
	      \begin{equation*}
	      \int_X g(x)\ d\mu(x) = \int_Y\int_{f^{-1}({y})}g(x)\ d\mu_y(x)df_\#\mu(y).
	      \end{equation*}
\end{itemize}
\end{thm}
This means in particular that any $\mu\in\PM(X\times Y)$ whose first marginal $\nu$ is given can be represented in this disintegrated way. 

On the other hand, whenever there is a measurable (in the sense of the first item above) family $\mu_x\in\PM(Y)$ given, for any $\nu\in\PM(X)$ the following formula defines a unique measure $\mu\in\PM(X\times Y)$:
\begin{equation*}
\mu(f) = \int_X\left(\int_Y f(x,y)\ d\mu_x(y)\right)\ d\nu(x),
\end{equation*}
with $f:X\times Y\rightarrow\R$ being a nonnegative measurable function. In this sense, disintegration can be seen as an opposite procedure to the construction of a product measure. 

%%%%%%%%%%%%%%%
%%%%%%%%%%%%%%%
%%%%%%%%%%%%%%%
%%%%%%%%%%%%%%%\input{Kapitel/Anhang}

\bibliographystyle{alpha}       % Set the bibliography style to AMS plain alphabetized. (Can use ``amsalpha'', ``abbrv'', "apalike", acm)
\bibliography{da}

\end{document}